\documentclass[11pt,amssymb,amsmath]{article}
\usepackage{amsfonts,mathrsfs,bbm,latexsym,rawfonts,amsmath,amssymb,amsthm,epsfig}
\usepackage{fullpage, setspace}
\usepackage{graphics,psfrag,subfigure}
 \newtheorem{thm}{Theorem}[section]

 \newtheorem{rem}{Remark}[section]
 \newtheorem{dfn}{Definition}[section]
 \numberwithin{equation}{section}

 \newcommand{\Real}{\mathbb{R}}

 \newcommand{\Torus}{\mathbb{T}}
 \newcommand{\Sphere}{\mathbb{S}}

 \newcommand{\p}{\partial}

 \newcommand{\ga}{\gamma}
 \newcommand{\om}{\omega}
 \newcommand{\ta}{\theta}
 \newcommand{\nn}{\nonumber}
 \newcommand{\nb}{\nabla}

 \def\lg{\langle} \def\rg{\rangle}

 \newcommand{\Sch}{Schr\"odinger~}

\begin{document}

\title{\bf Geometric solitons of Hamiltonian flows on manifolds}

\author{Chong Song\thanks{The first author is supported
by NSFC, Grant No. 11201387},\quad Xiaowei Sun\thanks{The second author
is supported by NSFC, Grant No. 11226082},\quad Youde Wang\thanks{The
third author is supported by NSFC, Grant No. 10990013}}

\maketitle

\begin{abstract}
It is well-known that the LIE(Locally Induction Equation) admit soliton-type solutions and same soliton solutions arise from different and apparently irrelevant physical models. By comparing the solitons of LIE and Killing magnetic geodesics, we observe that these solitons are essentially decided by two families of isometries of the domain and the target space respectively. With this insight, we propose the new concept of geometric solitons of Hamiltonian flows on manifolds, such as geometric \Sch flows and KdV flows for maps. Moreover, we give several examples of geometric solitons of the \Sch flow and geometric KdV flow, including magnetic curves as geometric \Sch solitons and explicit geometric KdV solitons on surfaces of revolution.
\end{abstract}

\tableofcontents
%-------------------------------------------------------------------------
\newpage

\section{Introduction}

The dynamics of vortex filaments has provided for almost a century
one of the most interesting connections between differential geometry
and soliton equations, and an example in which knotted curves arise as
solutions of differential equations possessing an infinite family of symmetries
and remarkably rich geometrical structures. The first mathematical model of the evolution
of an Eulerian vortex filament, an approximately one-dimensional region of the fluid where the velocity
distribution has a rotational component, was derived in 1906 by Luigi Da Rios \cite{D}, a student of Tullio
Levi-Civita. Both from mathematical and physical points of view, one has found that the
locally induction equation (also called Da Rios equation) is a fascinating nonlinear evolution equation.
It is related to the well-known Landau-Lifshitz (LL) equation which describes
the dynamics of several important physical systems such as ferromagnets, moving space curves, etc. and has intimate
connections with many of the well known integrable soliton equations, including nonlinear Schr\"{o}dinger and modified
KdV equations. It admits rich dynamical structures including solitons, dromions, vortices, etc.

A vortex filament is a curve $\gamma(t, s)$ in the Euclidean space $\Real^3$ depending on time $t$ and
arc-length parament $s\in \Real$. The evolution of $\gamma$ in Locally Induced Approximation (LIA) is
determined by the following Da Rios equation which is also called Locally Induction Equation (LIE):
\begin{equation}\nn
   \frac{\p\ga}{\p t} = \frac{\p\ga}{\p s}\times\frac{\p^2\ga}{\p s^2}, \quad \text{or} \quad \ga_t=\ga' \times \ga''.
\end{equation}
A careful derivation of the self-induced dynamics of a vortex filament can be found in G. K. Batchelor
\cite{B}. H. Hasimoto \cite{Hasimoto} discovered in 1972 that the LIE is equivalent to a complex-valued
cubic \Sch equation via the so-called Hasimoto transformation:
\begin{equation}\nn
    \Phi(\ga) = k \exp\left(i\int_0^s \tau ds'\right).
\end{equation}
where $k$ is the curvature and $\tau$ is the torsion of $\ga$.

It is well-known that the cubic \Sch equation defined on $\mathbb{R}$ admits a traveling wave soliton of the form
\begin{equation}\label{ee}
  \Phi(t, s) = \Phi_0(s-ct).
\end{equation}
From the correspondence between the solutions of
cubic \Sch equation and the solutions of LIE, it is not difficult to
see that (\ref{ee}) gives a soliton type solution $\ga(s, t)$ to the LIE whose curvature $k$ and torsion $\tau$ satisfy
\begin{equation}\label{e000}
    k(t, s) = k(s-ct), \quad \tau(t, s) = \tau(s-ct).
\end{equation}
Inserting (\ref{e000}) back into the cubic \Sch equation via
Hasimoto transformation, we reduce the LIE to an intrinsic soliton
equation of $k$ and $\tau$.

However, this intrinsic equation is not quite illuminating and the geometric significance is not obvious. Moreover, it is
always difficult to recover the curve $\ga$ from the intrinsic equation of $k$ and $\tau$. Hence,
a natural problem is how to understand the soliton equation from the viewpoint of global differential geometry
and reveal the hidden geometric structure of such solitons. We observe that the connection between solitons of
LIE and Killing magnetic geodesics sheds some light on this problem.

The so-called \emph{magnetic geodesic} is just the trajectory of a charged particle moving on a Riemannian manifold
with a magnetic field. The electron experiences a \emph{Lorentz force} $F(\,\cdot\,)$ and the magnetic geodesic $\ga$ satisfies equation
\begin{equation}\nn
    \nabla_{\ga'}\ga' = F(\ga').
\end{equation}
In particular, if the manifold is three dimensional, then the Lorentz force
is related to a vector field $V$ and the equation of magnetic geodesic is
\begin{equation}\label{e0000}
    \nabla_{\ga'}\ga' = V \wedge \ga',
\end{equation}
where $\wedge$ is the cross product induced by the Riemannian metric. If the vector $V$
is a Killing field, then a solution to the above equation is called a \emph{Killing magnetic geodesic}.

Many mathematicians and physicists have been interested in the existence of closed geodesics on a manifolds.
It seems that to study the existence of closed magnetic geodesics on a manifold is a more challenging problem. Even
for the case of two dimensional manifolds, it is a hard and long-standing problem and related to the problem of finding
curves with prescribed geodesic curvature. There is a large amount of works on this field using different
approaches in the last decades, and the results on existence of closed magnetic geodesics turn out to be subtle under various
assumptions. We refer to \cite{NT, G, Sch} and references therein.

On the other hand, the LIE can also be well-defined on
three dimensional manifolds by
\begin{equation}\label{e00000}
  \ga_t = \ga' \wedge \nabla_{\ga'}\ga'.
\end{equation}
In \cite{Barros}, the authors investigated the Killing magnetic geodesic in three dimensional space forms carefully and by
verifying the intrinsic equations, they found that the Killing magnetic geodesics are just solitons of the LIE. However,
this coincidence seems somewhat surprising and mysterious. In fact, one has also found that some other important but apparently irrelevant
physical models, such as elastic curves, share the same reduced equation for their respective solitons.
We refer to \cite{Fukumoto97} for more details and examples.

In this paper, our primary goal is to provide a more global and geometric view on the connection between solitons
of LIE and magnetic geodesics. The key observation is that the soliton of LIE is essentially
decided by two families of isometries. Namely, they are of the following form:
\begin{equation}\label{e0}
    \ga(s, t) = \phi_t\circ\ga_0\circ\psi_t(s),
\end{equation}
where $\phi_t$ is a one-parameter group of isometries of the ambient space, $\psi_t(s)$ is a one-parameter group of isometries of the domain, and $\ga_0$ is a curve independent of time $t$.

If we assume that $\phi_t$ is generated by a Killing vector field $V$, then by putting the form (\ref{e0}) back into the LIE (\ref{e00000}), we will get a reduced equation for $\ga_0$,
which is exactly equation (\ref{e0000}). In the specific setting of classical LIA, the ambient space is the Euclidean space $\Real^3$ and $\phi_t$ consists of translations and rotations,
while the domain is $\Real$ and $$\psi_t(s) = s-ct$$ is just a shifting of the parameter $s$ with speed $c$. Indeed, the soliton $\ga$ is generated by rotating and translating the initial
curve $\ga_0$, which turns out to be a Killing magnetic curve, with a slipping motion along the curve itself.

On the other hand, the LIE can be viewed as a Hamiltonian flow and there is an important generalization, namely, the \Sch flows for maps from
 a Riemannian manifold into a symplectic manifold \cite{BIKT,DW,Jurd,TU}. The \Sch flow is an important geometric flow which is closely related to the harmonic map and wave map. However, there seems to be no analog of soliton solution to the \Sch flow except, in some literature, the term
 ``soliton" refers to harmonic maps which is just a stationary solution of the flow equation. The significance of our observation is that a solution of the form (\ref{e0}) essentially gives a unified definition of solitons for geometric flows on manifolds.

More precisely, suppose $(M, g)$ is a Riemannian manifold and $(N, \om)$ is a symplectic manifold with a compatible almost complex structure $J$. The \Sch flow for a map $u:[0, T)\times M \to N$ is
\[\frac{\partial u}{\partial t} = J(u)\tau(u),\]
where $$\tau(u) := \mbox{trace}_g \nabla du$$ is the tension field. Then the soliton type solution of the \Sch flow analogous to (\ref{e0}) should has the following form:
\begin{equation}\label{e00}
  u(t, x) = \phi_t\circ u_0\circ \psi_t(x),
\end{equation}
where $\psi_t$ and $\phi_t$ are one-parameter groups of isometries
of $M$ and $N$ respectively, and $$u_0:M\to N$$ is a map independent
of time $t$. Since this kind of soliton solution is closely involved with
the geometry of the underlying manifolds, we
should call them ``\emph{geometric solitons}". (See
Definition~\ref{dfn:soliton} below for more details.) We will see
that many important and special models arising in mathematical
physics fall into the scale of geometric solitons, which provide us
solid evidences to justify the name and significance of geometric
soliton. In particular, we found that the geometric soliton of
1-dimensional \Sch flow on Riemannian surfaces turns out to be a magnetic geodesic,
and the well-known equivariant solutions of the
\Sch flow for maps from $\mathbb{R}^2$ or $S^2$ into a K\"ahler
manifold with axial symmetry (see \cite{GKT}) has the same structure as geometric solitons.

The geometric soliton can also be defined for the geometric KdV(Korteweg-de Vries) flow which is recently proposed by Sun and Wang \cite{SunW}. The geometric KdV flow also origins from the movement of vortex filaments, this time along with an axial flow, which transforms to a classical mKdV equation under the Hasimoto transformation. From the viewpoint of integrable systems, the geometric KdV flow from $\Real$ into $\Sphere^2$ belongs to the same hierarchy of the corresponding \Sch flow. Similarly, the geometric KdV soliton is defined as a solution of the geometric KdV flow in the form (\ref{e00}). In fact, this kind of geometric solitons can be defined for all similar Hamiltonian flows on manifolds.

The next goal is of course seeking the above geometric solitons. By inserting (\ref{e00}) back into the flow equation, we are led to study an equation of the generating map $u_0$ which is independent of the time variable $t$. For the \Sch flow, we get an elliptic system and $u_0$ turns out to be a harmonic map with potential. For the geometric KdV flow, we get a third-order o.d.e. equation. Each solution $u_0$ to the reduced equation generates a geometric soliton of the geometric flow. However, in general the reduced equation is very hard to solve except some special cases where the underlying manifolds are simple.

This paper is organized as follows. First we recall the classical traveling wave solutions of the LIE of vortex filament with or without axial flow, as well as elastic curves as solitons of LIE. Next we show the equivalence of these solitons and Killing magnetic geodesics extrinsically using our observation (\ref{e0}). Then after a brief review of the of \Sch flow and geometric KdV flow, we move on to give the definition of geometric solitons. Meanwhile, we derive a reduced equation for the geometric solitons. Finally, we show some examples when the underlying manifolds are simple. This includes magnetic geodesics as geometric \Sch solitons, and explicit examples of geometric KdV solitons on surfaces of revolution.

\section{Solitons of LIE}

\subsection{Localized Induction Approximation}

The motion of a thin vortex filament in an inviscid incompressible fluid is described asymptotically by the so-called \emph{Locally Induction Approximation(LIA)}. This is a minimum model which was first discovered by Da Rios in 1906:
\begin{equation}\label{equ:LIA}
  \ga_t = \ga' \times \ga''.
\end{equation}

If $\ga$ is parameterized by arclength and $k, \tau$ denotes the curvature and torsion of $\ga$ respectively, then under the Frenet frame $\{\mathbf{t},\mathbf{n},\mathbf{b}\}$, we have
\begin{equation*}
  \ga' = \mathbf{t}, \quad \mathbf{t}' = k \mathbf{n}, \quad \mathbf{n}' = -k\mathbf{t} + \tau\mathbf{b}.
\end{equation*}
Thus equation (\ref{equ:LIA}) becomes
\[ \ga_t = k \mathbf{b}. \]
\begin{rem}
  If $s$ is the arclength parameter of $\ga(0,\cdot)$ at the initial time $t =0$, it is easy to verify that $s$ is arclength parameter of $\ga(t, \cdot)$ for all time $t>0$.
\end{rem}

Hasimoto found a famous transformation which is now referred as the \emph{Hasimoto transformation}
\begin{equation}\label{Hasimoto}
    \Phi(\ga) = k \exp\left(i\int_0^s \tau ds'\right).
\end{equation}
Using this transformation, equation (\ref{equ:LIA}) collapses to a cubic \Sch equation for the complex valued function $\Phi$:
\begin{equation}\label{equ:sch}
    -i\Phi_t = \Phi_{ss} + \frac{1}{2}(|\Phi|^2 + A)\Phi,
\end{equation}
where $A$ is a real function of time variable $t$ (which can by eliminated by shifting the origin).

Note that the function $\Phi$ only depends on the intrinsic quantities of the curve $\ga$. Namely, $\Phi$ is given by the curvature
$k$ and torsion $\tau$ which is independent of the choice of coordinates and the parameter of $\ga$. Thus equation (\ref{equ:sch})
should be regarded as the intrinsic version of the LIE (\ref{equ:LIA}).

A remarkable feature of the cubic \Sch equation (\ref{equ:sch}) is that it admits a type of solitary wave solutions which demonstrate a kind of self-similarity. These solutions have the form
\begin{equation}\label{equ:travelling-wave}
  \Phi(t, s) = \Phi_0(s-ct)
\end{equation}
where $c$ is a constant describing the speed of the wave and $\Phi_0$ satisfies an elliptic equation:
\[ ic\Phi_0' = \Phi_0'' + \frac{1}{2}(|\Phi_0|^2+ A)\Phi_0. \]
Obviously, such solutions correspond to a type of solitary wave solution $\ga(t, s)$ of the LIE (\ref{equ:LIA}) with
\begin{equation}\label{equ:LIA-solitary0}
    k(t, s) = k(s-ct), \quad \tau(t, s) = \tau(s-ct).
\end{equation}
Putting this back to equation (\ref{equ:sch}), we get the intrinsic soliton equation for $k$ and $\tau$ as the following:
\begin{equation}\label{equ:LIA-solitary1}
\left\{
  \begin{aligned}
  & k'' + \frac{1}{2}k^3 - k\tau^2 + A(t)k = - ck(\tau - \tau(-ct)), \\
  & 2k'\tau + k\tau' = ck'.
  \end{aligned}
\right.
\end{equation}
In the next subsection, we will see that the intrinsic equation of elastic curves coincides with the above equation.

\subsection{Elastic curves as solitons of LIE}

An interesting and remarkable discovery noted by Hasimoto in \cite{Hasimoto} is that the soliton of LIE is related to elastic curves. The problem of elastic curve was first proposed by Bernoulli to Euler in 1744. In mathematics, the bending energy of a thin inextensible wire $\ga$ in 3 dimensional Euclidean space is described by the total squared curvature $$E(\ga) = \int_\ga k^2 ds,$$ where $k$ is the curvature of the curve $\ga$. A curve is called an \emph{elastic curve} if it minimizes the bending energy $E$.

In practice, we often use a modified energy which depends on the differentials of the curve explicitly. Namely, we define an elastic curve to be a critical point of the following energy:
\[ F(\ga) = \frac{1}{2}\int |\ga''|^2ds + \int \Lambda(s)(|\ga'|^2 - 1) ds, \]
where $\Lambda(s)$ is a point-wise Lagrangian multiplier, constraining the speed of the curve. Obviously, such an elastic curve in equilibrium satisfies the Euler-Lagrangian equation:
\[ \ga'''' - \frac{d}{ds}(\Lambda\ga') = 0. \]
Suppose $\ga$ is parameterized by arc-length and using the Frenet frame along the curve, the above equation can by rewrite as:
\begin{equation}\label{equ:elastica}
\left\{
  \begin{aligned}
  & k'' + \frac{1}{2}k^3 - k\tau^2 - \frac{\lambda}{2}k = 0, \\
  & 2k'\tau + k\tau' = 0.
  \end{aligned}
\right.
\end{equation}
Here $\lambda$ is a constant which is determined by the Lagrangian multiplier $$\Lambda = -\frac{3}{2}k^2 + \frac{\lambda}{2}.$$
The second equation integrates to give a point-wise constant quantity
\[k^2 \tau = c,\]
and the first equation becomes
\begin{equation}\label{e12}
  k^3k'' + \frac{1}{2}k^6 - \frac{\lambda}{2}k^4 - c^2 = 0
\end{equation}
With the help of elliptic functions, one can solve equation (\ref{e12}) to get (refer to \cite{Singer})
\begin{equation}\label{e122}
  k^2 = k_0^2\left(1-\frac{p^2}{w^2}\text{sn}^2\left(\frac{k_0}{2w}s, p\right)\right),
\end{equation}
where $\text{sn}(\cdot, p)$ denotes the elliptic sine function with parameter $p$, and the parameters $p$, $w$ and $k_0$
are related to constants $\lambda$ and $c$ satisfying $0\le p \le w \le 1$ and
\begin{equation}\nn
\left\{
  \begin{aligned}
  & 2\lambda w^2 = k_0^2(3w^2 - p^2 - 1),\\
  & 4c^2w^4 = k_0^6(1-w^2)(w^2-p^2).
  \end{aligned}
\right.
\end{equation}

Now let us recall the intrinsic equation (\ref{equ:LIA-solitary1}) for solitons of the LIE. If we introduce a new variable
$\xi := s -ct$ which is independent of $t$. Then the only two terms dependent of $t$ in the first equation of (\ref{equ:LIA-solitary1})
should cancel with each other. Namely, we have $$A(t)k - c k\tau(-ct) = A_1 k,$$ where $A_1$ is a constant. Comparing equation
(\ref{equ:LIA-solitary1}) to equation (\ref{equ:elastica}), one can easily find that a soliton of the LIE with $c = 0$
is exactly an elastic curve. Thus we get a special family of solitons of the LIE given by (\ref{e122}).

\begin{rem}
  Langer and Singer give a complete classification of elastic curves in $\Real^3$. The elastic curve may also be defined on 3 dimensional manifolds and similar results can be obtained. Refer to \cite{Singer} and references therein. Thus we may obtain many potential solitons of the LIE. However, it is somehow difficult to recover all the formations of possible solitons from the expression of $k$ and $\tau$. This was accomplished by Kida in \cite{Kida}.
\end{rem}

\subsection{Vortex filament with axial flow}

A vortex tube is often accompanied with an axial flow in the core. This phenomenon was modeled by Moore and Saffman \cite{MS} and then simplified by Fukumoto and Miyazaki \cite{Fukumoto91} as the following equation:
\begin{equation}\label{equ:LIA-axial}
    \ga_t = \alpha \ga' \times \ga'' + \beta [\ga''' + \frac{3}{2} \ga''\times(\ga'\times\ga'')],
\end{equation}
where $\alpha$ and $\beta$ are two constant numbers. In fact, the right hand of the above equation is a linear combination of the first two vector fields of the Localized Induction Hierarchy \cite{LP}. Under the Frenet frame, equation (\ref{equ:LIA-axial}) reads:
\begin{equation}\label{equ:LIA-axial2}
    \ga_t = \alpha k\mathbf{b} + \beta (\frac{1}{2}k^2\mathbf{t} + k'\mathbf{n} + k\tau\mathbf{b}),
\end{equation}

It's obvious that when $\alpha =1$, $\beta=0$, equation (\ref{equ:LIA-axial}) reduces to the LIE (\ref{equ:LIA}). On the other hand, if $\alpha=0$, $\beta=1$ the equation becomes:
\begin{equation}\label{equ:LIA-axial1}
    \ga_t = \ga''' + \frac{3}{2} \ga''\times(\ga'\times\ga''),
\end{equation}
Applying the Hasimoto transformation (\ref{Hasimoto}), equation (\ref{equ:LIA-axial1}) can be rewrite as a complex valued mKdV equation:
\begin{equation}\label{equ:KdV}
    \Phi_t = \Phi_{sss} + \frac{3}{2}|\Phi|^2\Phi_s.
\end{equation}
This is the second equation in the nonlinear \Sch hierarchy and it also admits a solitary wave soliton which satisfies (\ref{equ:travelling-wave}), and hence (\ref{equ:LIA-solitary0}). Following the same spirit, one can compute the intrinsic equation of a soliton type solution to equation (\ref{equ:LIA-axial1}):
\begin{equation}\label{equ:KdV-solitary1}
\left\{
  \begin{aligned}
  & 3k''\tau + 3k'\tau' + k\tau'' - k\tau^3 + \frac{3}{2}k^3\tau= - ck(\tau - \tau(-ct)), \\
  & k''' - 3k'\tau^2 - 3k\tau\tau' + \frac{3}{2}k^2k' = -ck'.
  \end{aligned}
\right.
\end{equation}

\section{Killing Magnetic geodesics}

\subsection{Killing magnetic geodesic}

A \emph{magnetic field} on a Riemannian manifold $(M, g)$ is characterized by a closed 2-form $\Omega \in \Lambda^2(M)$ which denotes the set of all smooth section of $\Lambda^2(T^*M)$.
A moving particle of charge in the presence of a magnetic field $F$ will experience a \emph{Lorentz force} $F$ which is represented by a skew-symmetric (1, 1) tensor defined by
\[ g(F(X), Y) = \Omega(X, Y), \quad \forall X, Y \in \Gamma(TM). \]
Here $\Gamma(TM)$ denotes the set of smooth vector fields on $M$.

The \emph{magnetic geodesic} $\ga$ describes the motion of an electron in a magnetic field which satisfies equation
\begin{equation}\label{equ:mg}
    \nabla_{\ga'}\ga' = F(\ga').
\end{equation}

In particular, if the manifold $M$ is 3 dimensional, then the 2-form $\Omega$ is equivalent to a vector field $V$ via the Riemannian metric $g$. Indeed, one can define $V = (\star \Omega)^\#$, where $\star$ is the Hodge star operator and $^\#$ denotes the promotion given by the metric $g$. Moreover, $d\Omega = 0$ implies
$$\mbox{div} (V) = \mbox{trace}_g(\nabla V)= 0.$$ On the other hand, we can also define a cross product using the volume form $\Omega_0$ by the formula:
\[ g(X\wedge Y, Z) = \Omega_0(X, Y, Z),  \quad \forall X, Y, Z \in \Gamma(TM). \]
In this case, it is easy to check that the Lorentz force is $F(X) = V\wedge X$. Thus the equation of magnetic geodesic (\ref{equ:mg}) becomes:
\begin{equation}\label{equ:kmg1}
    \nabla_{\ga'}\ga' = V \wedge \ga'.
\end{equation}
Especially if the vector field is a Killing field, that is, $V$ is a vector field which generates a family of 1-parameter group of isometries,
then a solution $\ga$ to equation (\ref{equ:mg}) is called a \emph{Killing magnetic geodesic}.  Note that the Killing magnetic curves can be explicitly determined, see \cite{MN, RM1, RM2}.

In \cite{Barros}, the authors discovered a Killing magnetic geodesic with arc-length parameter in 3 space form satisfies the following equivalent equation of (\ref{equ:kmg1}):
\begin{equation}\label{equ:kmg2}
\left\{
  \begin{aligned}
  & k'' + \frac{1}{2}k^3 - k\tau^2 + Ak = - \omega k\tau, \\
  & 2k'\tau + k\tau' = \omega k',
  \end{aligned}
\right.
\end{equation}
where $A$ is a constant number and $\omega$ is the so-called quasislope defined by
\[ \omega = \lg V, \mathbf{t} \rg = const.\]

It is easy to see that this intrinsic equation (\ref{equ:kmg2}) is exactly the same with the equation derived for the
solitary wave solution (\ref{equ:LIA-solitary1}) of LIE. The interesting connection between vortex filament and Killing
magnetic geodesic in three-dimensional space forms has been addressed in \cite{Barros}. But their observation is base on
comparing the intrinsic equations of the curvature and the torsion, which seems somewhat obscure and mysterious.
In the next subsection we are intended to explain the equivalence of Killing magnetic curves and solitons of LIE
from a more global and extrinsic view point, which seems more natural and illuminating. This will motivate the definition of geometric soliton in the next section.

\subsection{Relation to solitons of LIE}

Now we try to explain the equivalence of Killing magnetic curves in Euclidean space $\mathbb{R}^3$ and solitons of LIE.
Suppose $\ga(s, t)$ is a solitary wave solution of the LIE which comes from the soliton of the corresponding \Sch equation
via the Hasimoto transformation. Then the curvature and torsion of $\ga$ satisfies equation (\ref{equ:LIA-solitary0}), i.e.
$$k(s, t) = k(s-ct), \quad \tau(s, t) = \tau(s-ct).$$
We need to derive the underlying equation for the soliton $\ga$. Note that the above equation does \emph{not} necessarily imply that
the curve also satisfies $$\ga(s, t) = \ga(s -ct).$$ Indeed, both the curvature $k$ and the torsion $\tau$ are intrinsic
properties of the curve. Namely, they stay unchanged if the curve $\ga$ varies by an isometry, which corresponds to a translation and a rotation in the Euclidean space $\Real^3$. Therefore, we may suppose that there exists a one-parameter  group of isometries $\phi_t$ and a curve $\ga_0$ such that
\begin{equation}\label{equ:LIA-solitary2}
    \ga(s, t) = \phi_t(\ga_0(s -ct)).
\end{equation}
If the one-parameter  group of isometries $\phi_t$ is generated by a Killing vector field $V$, then we may compute
\begin{equation}\label{e1}
\begin{split}
  \frac{d}{dt}\ga(s,t) &= V(\phi_t(\ga_0)) + d\phi_t(\ga_0)(-c\ga_0')\\
  &= d\phi_t(V(\ga_0)) + d\phi_t(\ga_0)(-c\ga_0').
\end{split}
\end{equation}
On the other hand, we have
\begin{equation}\label{e2}
  \frac{d}{ds}\ga(s,t) = d\phi_t(\ga_0)(\ga_0')
\end{equation}
and
\begin{equation}\nn
  \frac{d^2}{ds^2}\ga(s,t) = d\phi_t(\ga_0)(\ga_0'') + d^2\phi_t(\ga_0)(\ga_0', \ga_0').
\end{equation}
Since $\phi_t$ is an isometries, we have $d^2\phi_t(\ga_0)(\ga_0', \ga_0') = 0$. Hence
\begin{equation}\label{e3}
    \frac{d^2}{ds^2}\ga(s,t) = d\phi_t(\ga_0)(\ga_0'').
\end{equation}
Putting (\ref{e1}), (\ref{e2}) and (\ref{e3}) back to the LIE (\ref{equ:LIA}), we can get
\begin{equation}\nn
    d\phi_t(\ga_0')\times d\phi_t(\ga_0'') = d\phi_t(V(\ga_0)) + d\phi_t(-c\ga_0'),
\end{equation}
which implies
\begin{equation}\label{equ:LIA-solitary3}
    \ga_0'\times \ga_0'' = V(\ga_0) - c\ga_0'.
\end{equation}
Comparing equation (\ref{equ:LIA-solitary3}) with equation (\ref{equ:kmg1}), we find out that they are equivalent.
That is, if $\ga$ is a solitary wave solution of the LIE (\ref{equ:LIA}), $\ga_0$ turns out to be exactly a Killing
magnetic geodesic in $\mathbb{R}^3$. Or inversely, every Killing magnetic geodesic $\ga_0$ in $\mathbb{R}^3$
generates a soliton type solution $\ga$ to the LIE.

\section{Geometric Solitons}

\subsection{Geometric Hamiltonian flows}

It is well-known that there exists a direct relation between the Landau-Lifshitz equation and the Da Rios equation. One can obtain
the Landau-Lifshitz equation by differentiating the Da Rios equation if the vortex filament curve is parameterized by arc-length parameter.
About a century after the discovery of the Da Rios equation, mathematicians began to realize that the Landau-Lifshitz equations stemming
from vortex filaments can be generalized to a Hamilton flow for maps from a Riemannian manifold into a symplectic manifold, i.e., the \Sch
map flows for maps into a symplectic manifold.

Recently, Sun and Wang \cite{SunW} proposed another geometric flow now referred as the geometric KdV flow which belongs to the same
hierarchy of the \Sch flow when the domain manifold is $\mathbb{R}$ and target is a compact Hermitian symmetric space \cite{TU}.
For the reader's convenience, we include here the definitions of these flows.

Suppose $(M, g)$ is a Riemannian manifold and $(N, \omega)$ is a Symplectic manifold. Assume $N$ is embedded in an Euclidean space $\Real^K$, $J$ is a tamed almost complex structure and $h$ is the corresponding Riemannian metric such that
\[ \omega(\xi, \eta) = h(\xi, J\eta), \forall \xi, \eta \in T_yN.\]
Let $W$ be the Sobolev space defined by
\[ W := \{u \in W^{1,2}(M, \Real^K) | u(x) \in N ~\text{for a.e.}~ x\in M\}. \]
Then the tangent space of $W$ at a map $u \in W$ is the pull-back tangent bundle $T_u W = u^*(TN)$. There is a natural symplectic form $\Omega$ on $W$ induced by $\omega$. Namely, for any $X, Y\in T_uW$,
\[ \Omega(X, Y) = \int_M \omega(X(u), Y(u)) dV_g. \]
Consider the energy functional
\[ E_1(u) = \frac{1}{2}\int_M |\nabla u|^2 dV_g \]
as a Hamiltonian function on $W$, then the associated Hamiltonian vector field $X_{E_1}$ satisfies
\[ \Omega(X_{E_1}, \cdot) = dE_1(\cdot). \]
Thus we can define the Hamiltonian flow by
\begin{equation}\label{e4}
    u_t = X_{E_1}(u) = J\nabla E_1(u).
\end{equation}
It is well-known that $$\tau_1(u) := \nabla E_1(u) = \Delta u + A(u)(\nabla u, \nabla u)$$ is the tension field where $A(u)$ denotes the second fundamental form. So the Hamiltonian flow (\ref{e4}) of $E_1$ is just the \Sch flow
\begin{equation}\label{equ:sch-flow}
    u_t = J\tau_1(u).
\end{equation}
When the dimension of $M$ is 1, namely, when $M = \Real^1$ or $M = S^1$, there is another important functional, namely, the pseudo-helicity, given by:
\[ E_2(u) = \frac{1}{2}\int_M \lg \nabla_x u_x, J u_x \rg. \]
Similarly, one can also define the Hamiltonian flow of $E_2$ by
\begin{equation}\label{equ:kdv-flow}
    u_t = J\tau_2(u),
\end{equation}
where $$ \tau_2(u) := \nabla E_2(u) = -J\nabla_x^2 u_x + \frac{1}{2} R(u_x, Ju_x)u_x$$
and $R$ is the curvature tensor of $N$. This is the geometric KdV flow defined by Wang and Sun recently, see \cite{SunW} for more details.

Next let's explain the relation between the above geometric flows and LIE.
Suppose $\gamma(x)$ is a closed curve in $\Real^3$ and is parameterized by arc-length, then $u:= \gamma'$ defines a map form $S^1$ into $S^2$ since $|u| = |\gamma'| =1$. Recall that there is the Frenet frame $\{\mathbf{t}, \mathbf{n}, \mathbf{b}\}$ along $\gamma$ in $\Real^3$ with $\mathbf{t} = \ga'$. On the other hand, there is an orthonormal basis $\{T, N\}$ on $\Sphere^2$ along $u$ where $T$ is the unit tangent vector along $u$ and $N$ is the normal vector. Note that the complex structure on $\Sphere^2$ is $J := u\times$ and we may suppose $N :=J T$. Then, we have
\[ u_x = \gamma'' = \mathbf{t}' = k \mathbf{n} = v T, \]
where $k$ is the curvature of $\gamma$ and $v:=|u_x|$ is the speed of $u$. Thus
\begin{equation}\label{e5}
    \mathbf{n} = T, ~ k = v.
\end{equation}
Furthermore,
\begin{equation*}
  \begin{split}
    \nabla_x u_x &= \nabla_{(vT)} (vT) = v_xT + v^2k_g N\\
     &= (\gamma''')^\top = (k'\mathbf{n} -k^2 \mathbf{t} + k\tau\mathbf{b})^\top\\
     & = k'\mathbf{n} + k\tau\mathbf{b},
  \end{split}
\end{equation*}
where $k_g$ is the geodesic curvature of $u$. Without loss of generality, we may assume that $ k \neq 0$, then
\begin{equation}\label{e6}
    \mathbf{b} = N, \quad\quad k_g = \frac {\tau}{k}.
\end{equation}

Now one can easily find that in this case, the Hamiltonian functions are just integrals of the curve $\gamma$ that arise from the study of elastic curves. Indeed,
\begin{equation*}
  \begin{split}
    &E_1(u) = \int_{S^1} |u_x|^2 ds = \int_{S^1} k^2 ds.\\
    &E_2(u) = \int_{S^1} \lg \nabla_x u_x, u\times u_x \rg ds = \int_{S^1} k^2\tau ds.
  \end{split}
\end{equation*}
On the other hand, after differentiating with respect to $x$, the Da Rias equation (\ref{equ:LIA}) becomes
\[ u_t = u \times u_{xx} \]
which is just the \Sch flow (\ref{equ:sch-flow}) on $S^2$. Moreover, the LIE of vortex filament with axial flow (\ref{equ:LIA-axial}) becomes
\[ u_t = \alpha u \times u_{xx} + \beta [u_{xxx} + \frac{3}{2}(u_x\times(u\times u_x))_x]. \]
One may easily check that the latter part of the above equation is exactly the geometric KdV flow (\ref{equ:kdv-flow}) on $S^2$.

\begin{rem}
  The geometric flow naturally connects the two apparently irrelevant but important physics models: elastic curves and vortex filaments.
  This relation remains valid in the case of three dimensional space forms, see \cite{Jurd}. The remarkable coincidence was stressed by
  many authors, but usually observed by comparing the intrinsic equations. Our perspective via the geometric flows provides a more global
  and geometric view. Namely, the LIA hierarchy is the Hamiltonian flow corresponds the bending energies of elastic curves and rods.
\end{rem}

\subsection{Geometric solitons}

Inspired by the relation between solitons of vortex filaments and Killing magnetic geodesics, we propose a new concept of geometric
soliton of the Hamiltonian flows, which can be regarded as a geometric generalization of the classical solitary wave solitons.

\begin{dfn}\label{dfn:soliton}
  Suppose $(M, g)$ is a Riemannian manifold and $(N, h, J)$ is a K\"ahler manifold. A solution $u$ to the Hamiltonian flow
  (\ref{equ:sch-flow}) or (\ref{equ:kdv-flow}) is called a geometric soliton, if there exist two one-parameter groups of
  isometries $\psi_t:M \to M$, $\phi_t:N \to N$ and a map $v: M \to N$ independent of $t$ such that $u$ has the form:
  \[  u(t, x) = \phi_t\circ v \circ \psi_t(x). \]
\end{dfn}

One can also define a geometric soliton which only involves the isometries of the target manifold. Namely, solitons
which have the form $u(t, x) = \phi_t \circ v(x)$ where $\phi_t$ is a family of isometries of the target manifold.
This kind of geometric solitons of the \Sch flow was first noted by Ding and Yin \cite{DY} in order to find periodic solutions
to the \Sch flow into $S^2$. Song and Wang \cite{SW} investigated geometric solitons of \Sch flow into Lorentzian manifolds
and successfully obtained soliton solutions of the hyperbolic Ishimori system. It is worthy to point out that the existence
of geometric solitons is conjectured to be related to the first eigenvalue of the underlying manifolds.

The concept of geometric solitons is a natural extension of the solitary wave solution to the LIE found by Hasimoto,
or the classical solitons of the \Sch equation and KdV equation. Indeed, a solitary wave solution of the LIE satisfies (\ref{equ:LIA-solitary2}), i.e.
\begin{equation*}
    \ga(s, t) = \phi_t(\ga_0(s -ct)),
\end{equation*}
where $\phi_t$ is the isometry of the $3$ dimensional Euclidean space consisting of translations and rotations.
The only difference at first glance is that the traveling speed $c$ of the solitary wave is missing in our definition
of geometric solitons. The effect of $c$ corresponds to a slipping motion along the curve itself. However, if we regard the transformation
$$\psi_t(s) = s - ct$$
as an isometries of the real line, the consistence of the definition of geometric soliton and classical solitary waves is justified. On the other hand, for a linear \Sch equation
\[ iu_t = \Delta u \]
defined on a flat torus $\Torus^m$, a solitary wave solution is of the form $u(t) = e^{ikt}v$ where $k$ is a positive constant and $v$ is a real function which satisfies the equation
\[ \Delta v + kv = 0. \]
Here, $\phi_t = e^{ikt}$ can be viewed as an isometric group and again this classical solitary wave solution falls in the scale of above defined geometric soliton.

Another important example is the well-known equivariant solution of \Sch flow on manifolds with axial symmetry. The $m$-equivariant map from $\Real^2$ to $\Sphere^2$ has the form
\[ u(r, \theta) = e^{m\theta R} v(r), \]
where $(r, \theta)$ are polar coordinates on $\Real^2$, $m$ is an integer, $v : [0,\infty) \to \Sphere^2$ is a map
independent of $t$, and $R$ is the matrix generating rotations around the $u^3$-axis. In a given homotopy class,
the $m$-equivariant harmonic map $Q_m$ minimizes the Dirichlet energy. A general problem on \Sch flow which has
attracted a considerable attention for the past ten years is to describe the flow for initial data near $Q_m$,
see for example \cite{GKT}. A major break-through on \Sch flow recently made in \cite{MRR} is the discovery
of a finite time blow up profile. The formation of singularity corresponds to the concentration of a universal bubble given by
\begin{equation}\label{e:bubble}
  u_0(x, t) = e^{\Theta(t)R}Q_1\left(\frac{x}{\lambda(t)}\right).
\end{equation}
If we regard $\phi_t = e^{\Theta(t)R}$ and $\psi_t(x) = \frac{x}{\lambda(t)}$ as one-parameter group of transformations, then the bubble (\ref{e:bubble}) has the same structure with our geometric soliton.

After modulo the isometries of the underlying manifolds, the only thing left is the time-independent map $v$. Just as the relation of solitons and Killing magnetic curves shown in the last section, the map $v$ satisfies a reduced equation. This is formulated in the following theorem.

\begin{thm}\label{thm1}
  Suppose the one-parameter groups of isometries $\{\phi_t\}, \{\psi_t\}$ are generated by a holomorphic Killing vector field
  $V \in \Gamma(TN)$ and a Killing vector field $W \in \Gamma(TM)$ respectively. Then\\
  i) $u = \phi_t\circ v \circ \psi_t$ is a geometric soliton of the \Sch flow (\ref{equ:sch-flow}) iff $v$ satisfies
  \begin{equation}\label{equ:sch-geometric-soliton}
    J\tau_1(v) = V(v) + \nabla_W v;
  \end{equation}
  ii) $u = \phi_t\circ v \circ \psi_t$ is a geometric soliton of the KdV flow (\ref{equ:kdv-flow}) iff $v$ satisfies
  \begin{equation}\label{equ:kdv-geometric-soliton}
    J\tau_2(v) = V(v) + \nabla_W v.
  \end{equation}
\end{thm}
\begin{proof}
Since the proof of the above two equalities are essentially the same, here we only prove i) for the geometric \Sch soliton and ii) follows.

The geometric \Sch soliton $u$ has the form
\[  u(t, x) = \phi_t\circ v \circ \psi_t(x), \]
where $\phi_t$ and $\psi_t$ are one-parameter groups of isometries generated by $V$ and $W$ respectively, that is,
\[ \frac{d\phi_t}{dt}(y) = V(\phi_t(y)), ~\forall y\in N \]
and
\[ \frac{d\psi_t}{dt}(x) = W(\psi_t(x)), ~\forall x\in M.\]
The group property of $\phi_t$ implies
\[ \phi_t \circ \phi_s = \phi_{s+t} = \phi_t \circ \phi_s. \]
Differentiating the above equality with respect to $s$ at $s=0$ gives
\[ d\phi_t(V) = V(\phi_t). \]
Similarly,
\[ d\psi_t(W) = W(\psi_t). \]
Hence by differentiating $u$ with respect to $t$, we have
\begin{equation*}
  \begin{split}
    \frac{\partial u}{\partial t}(t, x) &= V(\phi_t\circ v \circ \psi_t(x)) + d\phi_t\circ dv(W(\psi_t(x)))\\
    &= d\phi_t(V(v(\psi_t(x)))) + d\phi_t(\nabla_W v(\psi_t(x))).
  \end{split}
\end{equation*}
On the other hand, for any $t$ and $\bar{x}: = \psi_t(x)$, we have
\begin{equation*}
  \begin{split}
    \nabla^2 u(t, x)
    &= \nabla(d\phi_t(\nabla v(\bar{x})))\\
    &= d\phi_t(\nabla^2 v(\bar{x})) + d^2 \phi_t( \nabla v(\bar{x}),\nabla v(\bar{x})).
  \end{split}
\end{equation*}
Since $\phi_t$ is an isometries, the second fundamental form of $\phi_t$ is zero and it is obvious that the latter part of the above equality vanishes. Thus by taking trace, we get
\[ \tau_1(u(t, x)) = d\phi_t(\tau_1(v(\bar{x}))). \]
Recall the assumption that the family of isometries $\phi_t$ is holomorphic, which means $J\circ d\phi_t = d\phi_t \circ J$. Consequently, the \Sch flow $u_t = J\tau_1(u)$ now becomes
\[ d\phi_t(V(v(\bar{x})) + \nabla_W v(\bar{x})) = J\circ d\phi_t(\tau_1(v(\bar{x}))) = d\phi_t( J\tau_1(v(\bar{x}))), \]
which is equivalent to
\[ J\tau_1(v) = V(v) + \nabla_W v. \]
\end{proof}

\section{Examples}

\subsection{Geometric \Sch solitons on surfaces}

Here we discuss the special case of one dimensional geometric \Sch solitons on surfaces. Suppose $(M^2, g)$ is a Riemann surface and $$u = \phi_t\circ\ga\circ\psi_t: \Real^1\times[0, T) \to M^2$$ is a geometric \Sch soliton, where $\phi_t$ and $\psi_t$ are one-parameter groups of isometries and $\ga$ is a curve with arc-length parameter $s$. If $\phi_t$ is generated by a Killing vector field $V$ and $\psi_t = s-ct$, then by Theorem~\ref{thm1}, the generating curve $\ga$ satisfies the reduced equation:
\begin{equation}\label{equ:1}
  J(\ga)\nabla_s\ga_s = V(\ga) - c\ga_s,
\end{equation}
where $J$ is the complex structure of $M^2$. Since $\ga$ is parameterized by arc-length, we may denote
$$T := \ga_s,\quad N :=JT \quad \text{and}\quad  \nabla_s\ga_s = k_gN,$$ where $k_g$ is the geodesic curvature of $\ga$. Then from the equation (\ref{equ:1}) we may deduce that there exists a function $q$ such that $V(\ga) = q(\ga)\ga_s$. Therefore equation (\ref{equ:1}) becomes
\begin{equation}\label{equ:2}
  \nabla_s\ga_s = (q(\ga)-c)J(\ga)\ga_s.
\end{equation}

The above equation shows that the curve $\ga$ is a magnetic geodesic on $M^2$. In fact, on a Riemannian surface $M^2$, a magnetic field $\Omega$ has the form $\Omega = k\Omega_0$ where $\Omega_0$ is the volume form of $M^2$ and $k$ is a function. Thus the corresponding Lorentz force $F$ is $F(X) = kJX$ and the equation (\ref{equ:mg}) for magnetic geodesic becomes
\begin{equation}\label{equ:3}
  \nabla_s\ga_s = k(\ga)J(\ga)\ga_s.
\end{equation}
Comparing equation (\ref{equ:2}) with (\ref{equ:3}), we may find the \Sch soliton $\ga$ is just a magnetic geodesic with $$k(\ga) = q(\ga) - c.$$

It is obvious that a magnetic curve on $M^2$ with arc-length parameter has geodesic curvature $k$. Thus magnetic geodesic is closely related to the problem of prescribed geodesic curvature. In particular, a geometric \Sch soliton with trivial Killing vector field $V = 0$ is a curve with constant geodesic curvature $c$.

It has been a long-standing problem in the last decades to find closed magnetic geodesics or curves with prescribed geodesic curvature on surfaces which was initiated by Arnold \cite{A}. The existence turns out to be subtle and somehow related to the topology and geometry of the underlying manifolds. Various results are obtained by using different approaches, including Morse-Novikov theory and symplectic methods, etc. For example, Schnieder \cite{Sch} successfully applied an infinite-dimensional index theorem to prove that there exist at least two simple closed magnetic curves on $S^2$ if certain geometric assumptions are assumed. This would certainly yield the existence of two geometric \Sch solitons on $S^2$. For more results on closed magnetic curves, we refer to \cite{G, NT, Sch} and references therein.

\subsection{Geometric KdV solitons on surfaces of revolution}

Next we deduce the equation of geometric KdV soliton surfaces of revolution in the three-dimensional Euclidean space
$\mathbb{R}^3$.

To begin with, let $\Sigma$ be a surface of revolution in
$\mathbb{R}^3$ with coordinates $z_1,z_2$ and $z_3$. Without loss of
generality, assume that $\Sigma$ is obtained by revolving a regular
plane curve $\gamma(r)$ in the plane $z_2=0$ around the $z_3$-axis
with $r$ as its arc-length parameter, say $(f(r),0,g(r))$. Then
$\Sigma$ could be expressed as
$$\Sigma=:(f(r)\cos \theta, f(r)\sin\theta,g(r)),$$
with $${f_r}^2+{g_r}^2=1,$$ where $f_r$, $g_r$ are the derivatives of
$f(r)$ and $g(r)$ respect to $r$ respectively.

Let $\phi_t(z):\mathbb{R}\times\Sigma\rightarrow \Sigma$ be the obvious one-parameter group of isometries on $\Sigma$ which is just a rotation around the $z_3$-axis with speed $\om$ given by
$$\phi_t(z)=\left(
                                  \begin{array}{ccc}
                                    \cos \om t & -\sin \om t & 0 \\
                                    \sin\om t & \cos\om t & 0 \\
                                    0 & 0 & 1 \\
                                  \end{array}
                                \right)\left(
                                         \begin{array}{c}
                                           z_1 \\
                                           z_2 \\
                                           z_3 \\
                                         \end{array}
                                       \right).
$$
Then the corresponding Killing vector field on $\Sigma$ is
$$V(p)=\omega f(-\sin\theta, \cos\theta, 0)^T, \ p = (f(r)\cos \theta, f(r)\sin\theta,g(r)) \in \Sigma.$$
Moreover, let $\psi_t(x)=x+ct$ be the one-parameter groups of isometries of $\Real$ with constant speed $c$, then the
corresponding Killing field $W$ is just the constant $c$. By Definition \ref{dfn:soliton}, a geometric KdV soliton
$\widetilde{u}(x,t): \mathbb{R}\times[0,+\infty)\rightarrow \Sigma$
on $\Sigma$ has the form $\widetilde{u}(x,t)=\phi_t\circ u \circ
\psi_t$ where $u = u(x):\mathbb{R}\rightarrow \Sigma$ is the initial map. For simplicity, we consider the 1-equivariant map $u$ with $\theta=x$ and $r=r(x)$. Then by Theorem~\ref{thm1} the initial map
$$u(x)=(f(r(x))\cos x,f(r(x))\sin x, g(r(x)) ),$$
satisfies the reduce equation
$$\nabla_x^2 u_x + \frac{1}{2}R(u_x, Ju_x)Ju_x = V(u)+\nb_W u.$$
Here $\nb_W u = cu_x$ and
$$u_x=(r'f_r \cos x-f\sin x  , r'f_r\sin x+f\cos x  , r'g_r),$$
where $r', \ta'$ denotes the derivative respect to $x$. Since $\Sigma$ is a surface, it is easy to check that
$$R(u_x, Ju_x)Ju_x=G(u)|u_x|^2u_x,$$
where $G(u)$ is the Gauss curvature of $u(x)$ on $\Sigma$. Note that on the surface of revolution $\Sigma$ the Gauss curvature
$$G(u) = G(f) = -\frac{f_{rr}}{f}$$
only depends on $f$, which means $G$ is constant along a parallel. Now we may focus on the equation
\begin{equation}\label{e7}
\nabla_x^2 u_x + \frac{G(u)}{2}|u_x|^2u_x= V(u)+cu_x.
\end{equation}

Let $v:=|u_x|$ be the speed of $u(x)$, $T={u_x\over |u_x|}$ be
the unit tangent vector of $u(x)$ and $N\in T\Sigma$ the unit
tangent vector orthogonal to $T$. Then we have
$$v^2=|u_x|^2=r'^2+f^2,$$
and
$$\nb_x u_x = \nb_x(v T)=v'T+v^2 k_g N;$$
\begin{eqnarray}\label{e8}
\nb_x^2u_x&=&\nb_x(v'T+v^2 k_g N)\nn\\
&=&(v''-v^3k_g^2)T+(3vv'k_g+v^2k_g')N,
\end{eqnarray}
where $k_g$ is the geodesic curvature.  Here we used the formula
$$\nb_TT=k_gN\ \text{and}\ \nb_TN=-k_gT.$$
Substituting (\ref{e8}) into (\ref{e7}) yields
\begin{eqnarray}\label{e9}
(v''-v^3k_g^2+{G(f)\over 2}v^3-c)T+(3vv'k_g+v^2k_g')N=\left(
                  \begin{array}{c}
                    -\om f\sin\ x \\
                    \om f\cos\ x \\
                    0\\
                  \end{array}
                \right).
\end{eqnarray}
Let $$V(u)= k_1T+k_2 N,$$ it is easy to check that
\begin{equation}\label{e11}
k_1=V(u)\cdot T={\om f^2 \over v}, \quad
k_2=V(u)\cdot N=-{\om fr'\over v}.
\end{equation}
In view of (\ref{e11}) and (\ref{e9}) we get the equations of
geometric KdV soliton as follows
\begin{equation}\label{e10}
\left\{
\begin{aligned}
&vv''-v^4k_g^2+{G(f)\over 2}v^4-cv=\om f^2;\\
&3v^2v'k_g+v^3k_g'=-\om f r';\\
&v^2=r'^2+f^2.
\end{aligned}\right.
\end{equation}

\noindent We are interested in the following special cases:\\

i) If the KdV soliton has constant speed, i.e. $v=v_0,$ then (\ref{e10}) becomes
\begin{equation}\label{e14}
\left\{
\begin{aligned}
&-v_0^4k_g^2+{G(f)\over 2}v_0^4-cv_0=\om f^2;\\
&v_0^3k_g'=-\om f r';\\
&r'^2+f^2=v_0^2.
\end{aligned}\right.\nn
\end{equation}

ii) If $\phi_t = id$, i.e. the Killing field $V=0$ and $\om=0$, then the geometric KdV
soliton is just a KdV curve studied in \cite{Song} and
\begin{equation}\nn
\left\{
\begin{aligned}
&vv''-v^4k_g^2+{G(f)\over 2}v^4-cv=0;\\
&3v^2v'k_g+v^3k_g'=0.
\end{aligned}\right.
\end{equation}
Thus $(v^3k_g)'=0$ which implies $v^3k_g=C$ is a constant. Thus we have
$$v^3v''+{G(f)\over2}v^6 -cv^3=C^2.$$
Moreover, in this case, if $u(x)$ has a constant speed, i.e.,
$v=v_0$ is constant, then the geodesic curvature on the KdV curve
must be constant and so is Gauss curvature. In another word, the curve should be a parallel of $\Sigma$ with $r(x) = r_0.$

iii) Inversely if we assume the KdV soliton is a  parallel of $\Sigma$, i.e. $r(x) = r_0$ is a constant, then
both $f=f(r_0)$ and $G = G(f)$ are constants. The equation (\ref{e10}) yields
\begin{equation}\nn
\left\{
\begin{aligned}
&vv''-v^4k_g^2+{G\over 2}v^4-cv=\om f^2;\\
&3v^2v'k_g+v^3k_g'=0;\\
&v^2=f^2.
\end{aligned}\right.
\end{equation}
which implies that $v =f$ is constant. By the same arguments, the quantity $v^3k_g=C$ is constant and so is the geodesic curvature $k_g$. Therefore equation (\ref{e10}) reduces to an equation of $f(r_0)$ and $f^{''}(r_0)$:
$${1\over2}f^{''}f^5+\om f^4 + cf^3 +C^2 = 0.$$
Thus given a surface of revolution and a constant $C$, we can use the above equation to identify a parallel which is a geometric KdV soliton with $v^3k_g=C$.

\subsection{Geometric KdV solitons on cylinders}

In order to see the geometric KdV soliton more precisely, we now show explicit examples
on a cylinder. Suppose $r>0$ is a constant and
$$\Sigma=(r\cos \theta, r\sin \theta, h ), \ \ \theta \in [0, 2\pi], \ h \in \Real$$
is a cylinder of radius $r$.
There are two natural Killing fields on the cylinder generated by the rotation and the translation along the $z_3$-axis respectively.
Namely, a Killing field on $\Sigma$ is the combination of a rotation and a translation and has the form:
$$V = \om R + \sigma T$$
where $R = (-\sin \theta, \cos \theta, 0)^T$, $T = (0, 0 ,1)^T$ and $\omega, \sigma \in \Real$.
 The corresponding one-parameter group of isometries
$\phi_t(z):\mathbb{R}\times\Sigma\rightarrow \Sigma$ is given by
$$\phi_t(z)=\left(
                                  \begin{array}{ccc}
                                    \cos \om t & -\sin \om t & 0 \\
                                    \sin\om t & \cos\om t & 0 \\
                                    0 & 0 & 1 \\
                                  \end{array}
                                \right)\left(
                                         \begin{array}{c}
                                           z_1 \\
                                           z_2 \\
                                           z_3 \\
                                         \end{array}
                                       \right)+ \left(
                                                  \begin{array}{c}
                                                    z_1 \\
                                                    z_2 \\
                                                    z_3+\sigma t \\
                                                  \end{array}
                                                \right)
                                       $$
 Let $u(x)=(r\cos kx, r\sin kx, h(x)), x\in \Real$ be a curve on $\Sigma$, where $k \in \Real$ is some constant. The metric on a cylinder is flat and the curvature vanishes. Thus the equation of geometric KdV soliton on $\Sigma$ is just
$$u_{xxx}=V(u),$$
or more precisely,
$$\left(
    \begin{array}{c}
      rk^3\sin x \\
      -rk^3\cos x \\
      h'''(x) \\
    \end{array}
  \right)=\om\left(
                \begin{array}{c}
                  -\sin x \\
                  \cos x \\
                  0 \\
                \end{array}
              \right)+\sigma\left(
                        \begin{array}{c}
                          0 \\
                          0 \\
                          1\\
                        \end{array}
                      \right).
$$
This implies $\om=rk^3$  and $h'''(x)=\sigma.$ Hence we obtain that
$$h(x)={\sigma\over6}x^3+C_3x^2+C_2x+C_1,$$
where $C_i(i=1,2,3)$ could be any constants.

If $\sigma=0$, then we obtain rotating solitons. In this case, $h(x)$ is a linear combination of
$x^2,$ $x$ and $1$. Obviously, if $h(x)=C_1$, the curves are just
parallels on the cylinder; if $h(x)=C_2x,
(C_2\neq0),$ then we get helixes; while if
$h(x)=C_3x^2,(C_2\neq0)$, the geometric solitons are parabolic
curves rolled on the cylinder. If $\sigma\neq0$, the one-parameter isometric groups
contains a translation along the $z_3$-axis. These KdV solitons are cubic curves on the cylinder.
In fact, the rotating KdV solitons are
combinations of these three kinds of curves. (See Figure 1)

\begin{figure}[!htbp]
  \centering
  \subfigure[$h(x)=C_2x$]{
    \label{fig:subfig:a} %% label for first subfigure
    \includegraphics[width=5cm,height=4cm]{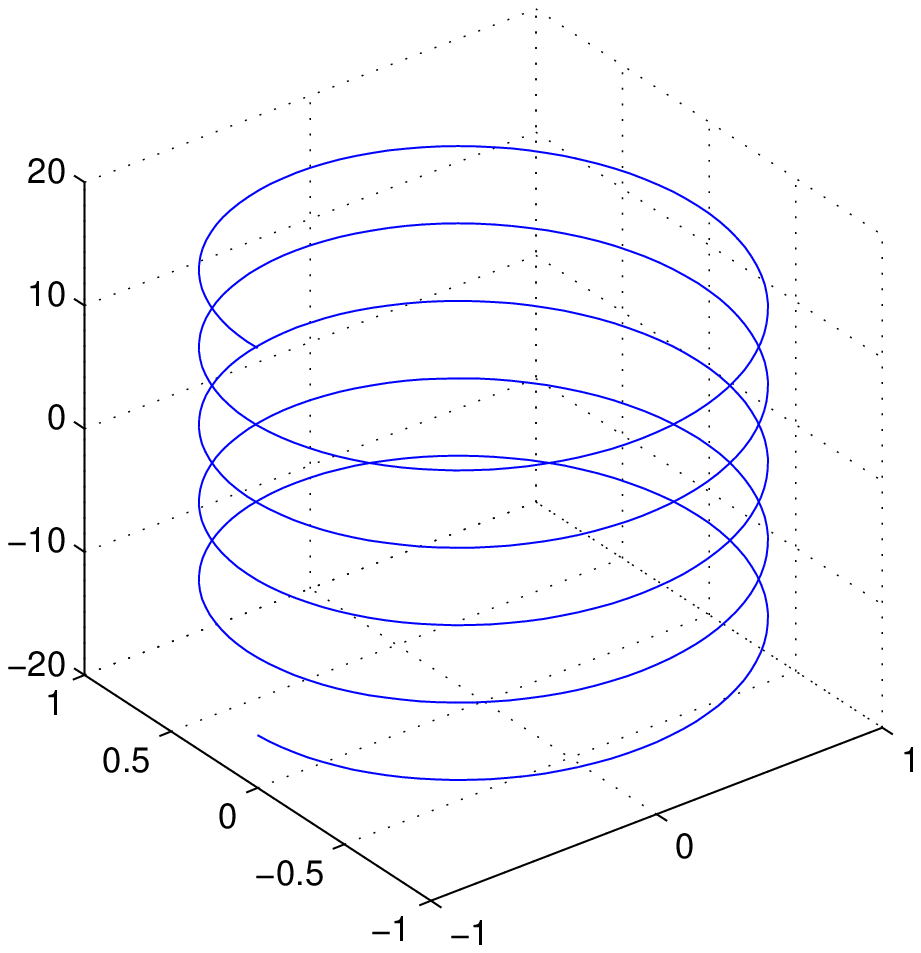}}
  %\hspace{0.1in}
  \subfigure[$h(x)=C_3x^2$]{
    \label{fig:subfig:b} %% label for second subfigure
    \includegraphics[width=5cm,height=4cm]{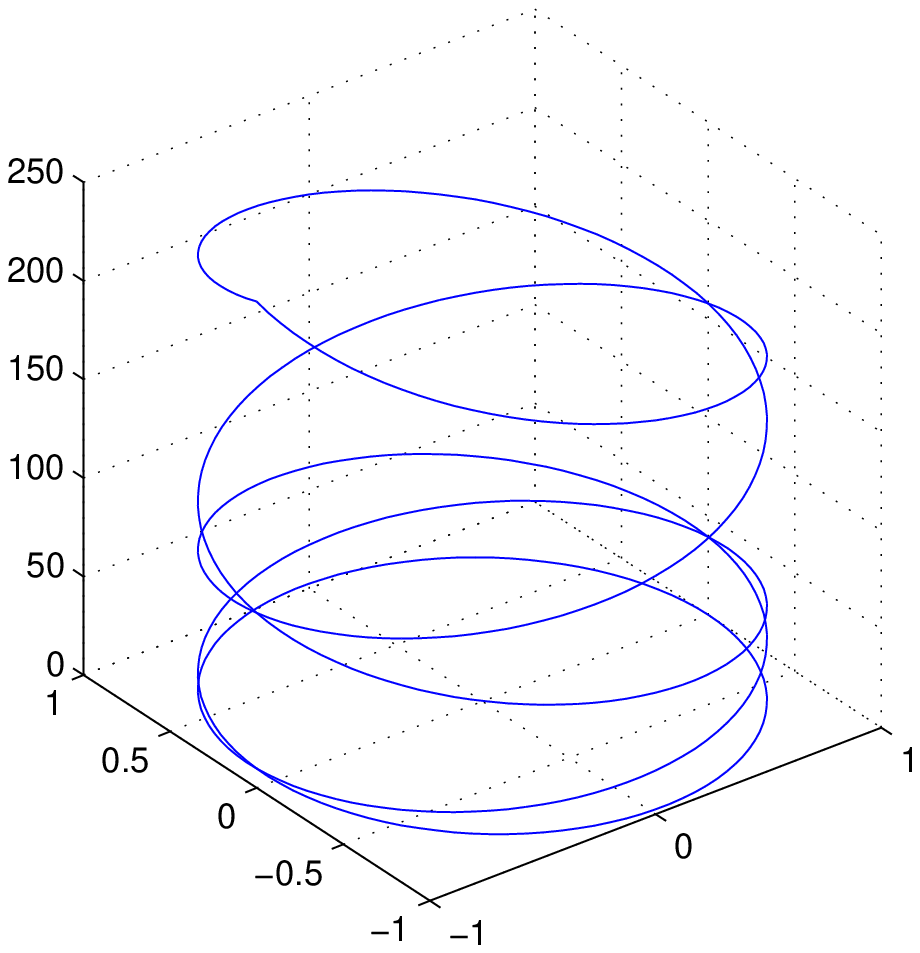}}
     %\hspace{0.1in}
  \subfigure[$h(x)=C_4x^3$]{
    \label{fig:subfig:c} %% label for second subfigure
    \includegraphics[width=5cm,height=4cm]{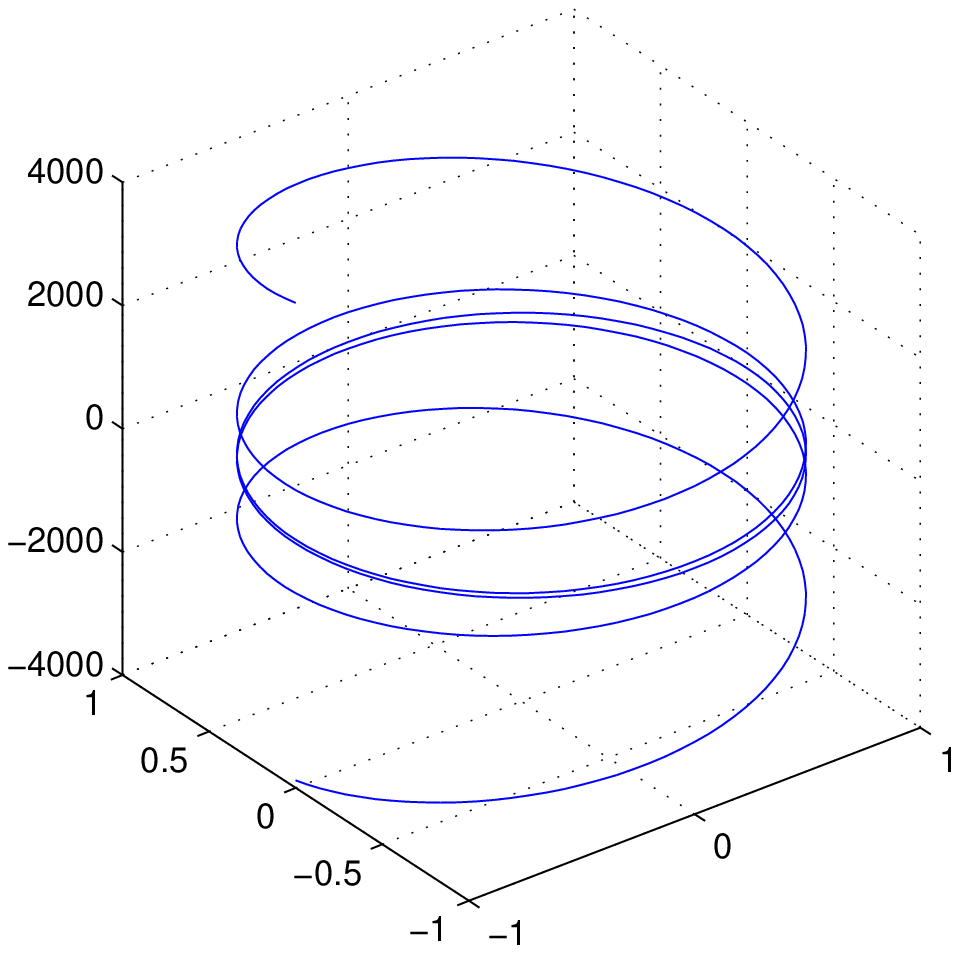}}
      \caption{Geometric KdV Solitons on Cylinder.}
  \label{fig:subfig} %% label for entire figure
\end{figure}

% ------------------------------------------------------------------------
%the bibliography
%\begin{thebibliography}{20}

\vspace{1cm}
\noindent{Chong Song}\\
School of Mathematical Sciences, Xiamen University\\
Xiamen 361005, P.R. China.\\
Email: songchong@xmu.edu.cn\\\\
\noindent{Xiaowei Sun}\\
School of Applied Mathematics, Central University of Finance and Ecnomics\\
Beijing 100081, P.R. China.\\
Email: sunxw@cufe.edu.cn\\\\
Youde Wang\\
Academy of Mathematics and Systems Science\\
Chinese Academy of Sciences,\\
Beijing 100190, P.R. China.\\
Email: wyd@math.ac.cn

\end{document}